\documentclass[12pt,a4paper]{amsart}
\usepackage{amsfonts}
\usepackage{esint} \usepackage{amsthm}
\usepackage{amsmath}
\usepackage{xcolor}
\usepackage{amssymb}
\usepackage{mathtools} \usepackage{amscd}
\usepackage{dsfont} \usepackage[latin2]{inputenc}
\usepackage{t1enc}
\usepackage[mathscr]{eucal}
\usepackage{indentfirst}
\usepackage{graphicx}
\usepackage{graphics}
\usepackage{epic}
\numberwithin{equation}{section}
\usepackage[margin=2.9cm, marginparwidth=2.4cm]{geometry}
\usepackage{epstopdf} 
\usepackage{enumerate}
\usepackage{mathrsfs} \usepackage{cases}
\usepackage{enumitem}

\usepackage[colorlinks=true, pdfstartview=FitV, linkcolor=blue, citecolor=blue, urlcolor=blue,pagebackref=false]{hyperref}

\usepackage{microtype}

\usepackage{tikz}
\usepackage{tkz-euclide}

\theoremstyle{plain}
\newtheorem{Th}{Theorem}[section]
\newtheorem{Lemma}[Th]{Lemma}
\newtheorem{Cor}[Th]{Corollary}
\newtheorem{Prop}[Th]{Proposition}
\newtheorem{proposition}[Th]{Proposition}

 \theoremstyle{definition}
\newtheorem{Def}[Th]{Definition}
\newtheorem*{Def*}{Definition}

\newtheorem*{Rem*}{Remark}
\newtheorem{?}[Th]{Problem}

\renewcommand{\P}{\mathsf{P}}

\newcommand{\rmk}[1]{\textcolor{black}{#1}}

\newcommand{\R}{\mathbb{R}}
\newcommand{\N}{\mathbb{N}}

\newcommand{\la}{\langle}\newcommand{\ra}{\rangle}
\newcommand{\eps}{\epsilon}

\renewcommand{\H}{\mathcal{H}}

\newcommand{\cH}{\mathcal{H}}

\renewcommand{\S}{\mathbf{S}}

\newcommand{\tr}{\mathsf{tr}}
\newcommand{\nn}{\mathbf{n}}

\newcommand{\cL}{\mathcal{L}}

\newcommand{\F}{\mathcal{F}}
\newcommand{\C}{\mathcal{C}}
\newcommand{\itr}{\mathsf{int}\,}
\newcommand{\cl}{\mathsf{cl}\,}
\newcommand{\dom}{\mathsf{dom}\,}
\newcommand{\bd}{\mathsf{bd}\,}

\newcommand{\cspan}{\overline{\mathsf{span}}\,}
\renewcommand{\odot}{\vee}
\newcommand{\lspan}{\mathsf{span}\,}
\newcommand{\E}{\mathcal{E}}
\newcommand{\rank}{\mathsf{rank}}

\newcommand{\diag}{\mathsf{diag}}

\usepackage{ulem}
\usepackage{cancel}

\begin{document}

\author[Hong-Bin Chen]{Hong-Bin Chen}
\address[Hong-Bin Chen]{Courant Institute of Mathematical Sciences, New York University, New York, New York, USA}
\email{hbchen@cims.nyu.edu}

\author[Jiaming Xia]{Jiaming Xia}
\address[Jiaming Xia]{Department of Mathematics, University of Pennsylvania, Philadelphia, Pennsylvania, USA}
\email{xiajiam@sas.upenn.edu}


\title{Fenchel--Moreau identities on \rmk{convex} cones}

\begin{abstract}
    A pointed convex cone naturally induces a partial order, and further a notion of nondecreasingness for functions. We consider extended real-valued functions defined on the cone. Monotone conjugates for these functions can be defined in an analogous way to the standard convex conjugate. The only difference is that the supremum is taken over the cone instead of the entire space. We give sufficient conditions for the cone under which the corresponding Fenchel--Moreau biconjugation identity holds for proper, convex, lower semicontinuous, and nondecreasing functions defined on the cone. In addition, we show that these conditions are satisfied by a class of cones known as perfect cones.
\end{abstract}

\maketitle

\section{Introduction}

The classical Fenchel--Moreau identity can be stated as $f=f^{**}$ for convex $f:\H\to(-\infty, \infty]$ satisfying a few additional regularity conditions. Here $\H$ is a Hilbert space with inner product $\la\cdot,\cdot\ra$ and the convex conjugate is given by
\begin{align*}
    f^*(x) = \sup_{y\in\H}\{\la y, x\ra - f(y)\},\quad \forall x\in \H.
\end{align*}
Note that the supremum is taken over the entire space $\H$. 

\smallskip

On the other hand, it is well-known (c.f. \cite[Theorem 12.4]{rockafellar1970convex}) that if $f:[0,\infty)^d\to(-\infty,\infty]$ is convex with extra usual assumptions and, in addition, is nondecreasing in the sense that
\begin{align*}
    f(x)\geq f(y),\quad \text{ if } x -y \in [0,\infty)^d,
\end{align*}
then we also have $f=f^{**}$. Here $*$ stands for the monotone conjugate defined by
\begin{align*}
    f^*(x)= \sup_{y\in[0,\infty)^d}\{\la y,x\ra -f(y)\},\quad \forall x\in [0,\infty)^d.
\end{align*}
The inner product appearing above is the standard one in $\R^d$. The nonnegative orthant $[0,\infty)^d$ is a cone in $\R^d$ and the nondecreasingness can be formulated with respect to the partial order induced by this cone. Compared with the convex conjugate, the supremum above is taken over the cone.

\smallskip 

Recently, in \cite{chen2020hamilton}, to study a certain Hamilton--Jacobi equation with spatial variables in the set of $n\times n$ (symmetric) positive semidefinite (p.s.d.) matrices denoted by $\S^n_+$, a version of the Fenchel--Moreau identity on $\S^n_+$ is needed to verify that the unique solution admits a variational formula. The derivation of such formulae for Hamilton--Jacobi equations on entire Euclidean spaces are known and can be seen, for instance, in \cite{bardi1984hopf,lions1986hopf}. On $\S^n_+$, \cite[Proposition~B.1]{chen2020hamilton} proves that $f=f^{**}$ holds if $f:\S^n_+\to(-\infty,\infty]$ is convex with some usual regularity assumptions and is nondecreasing in the sense that
\begin{align*}
    f(x)\geq f(y),\quad \text{ if }x-y\in\S^n_+.
\end{align*}
Accordingly, here $*$ stands for the monotone conjugate with respect to $\S^n_+$ given by
\begin{align*}
    f^*(x)=\sup_{y\in\S^n_+}\{\la y,x\ra - f(y)\},\quad\forall x\in\S^n_+.
\end{align*}
The inner product is the Frobenius inner product for matrices. Again, in this case, $\S^n_+$ can be viewed as a cone in $\S^n$, the space of $n\times n $ real symmetric matrices.

\smallskip

In view of these two examples, it is natural to pursue a generalization to an arbitrary (convex) cone $\C$ in a Hilbert space $\H$. More precisely, we want to show $f=f^{**}$ for proper, lower semicontinuous and convex $f:\C\to(-\infty, \infty]$ which is also nondecreasing in the sense that
\begin{align*}
    f(x)\geq f(y),\quad \text{ if }x-y\in\C,
\end{align*}
where
\begin{align}
    f^*(y) & = \sup_{z\in\C}\{\la z,y\ra - f(z)\},\quad \forall y\in\C^\odot, \label{eq:def_f*}
    \\
    f^{**}(x)&=\sup_{y\in\C^\odot}\{\la y,x\ra-f^*(y)\},\quad \forall x\in\C,\label{eq:f**}
\end{align}
where $\C^\odot$ is the dual cone of $\C$.

In Theorem~\ref{th:gen}, we give sufficient conditions on $\C$ for $f=f^{**}$ to hold for all $f$ satisfying the aforementioned properties. In particular, these conditions hold for a class of cones called \textit{perfect cones} first introduced in \cite{barker1978perfect} in the setting of Euclidean spaces. In short, a perfect cone is a self-dual cone satisfying that every face $\F$ of $\C$ is self-dual in the linear space spanned by $\F$.

\smallskip

\smallskip

The nonnegative orthant $[0,\infty)^d$ and the set of p.s.d.\ matrices $\S^n_+$ are both perfect cones. The former is easy to see using Definition~\ref{def:perfect} and the latter will be proved in \rmk{Proposition}~\ref{lemma:psd_perfect}. An example of an infinite-dimensional perfect cone is given in Lemma~\ref{lemma:circ_perfect}. Classical references for properties of cones and self-dual cones in Euclidean spaces or Hilbert spaces include \cite{barker1978faces,barker1981theory,barker1976self,bellissard1978homogeneous,penney1976self}. The generality pursued in this work is also motivated by the study of Hamilton--Jacobi equations arising in mean-field disordered systems \cite{mourrat2018hamilton, mourrat2019hamilton, mourrat2019parisi,mourrat2020nonconvex,mourrat2020free,chen2020hamilton_nonsym,chen2020hamilton}, where the solution is defined on a set that can be identified with a cone in possibly infinite dimensions, and expected to be nondecreasing with respect to the cone.

\smallskip

Let us briefly comment on the connection to the theory of abstract convexity and related works. Let $\mathscr A$ be the collection of affine functions with slopes in $\C^\vee$. In view of~\eqref{eq:def_f*} and~\eqref{eq:f**}, we can declare a function $f$ on $\C$ to be $\mathscr A$-convex if $f$ is equal to the upper envelope of all functions in $\mathscr A$ lying below $f$ (see \eqref{eq:sup_set} and the right-hand side of \eqref{eq:upper_envelope}). Then, by the Fenchel--Moreau theorem for abstract convexity (c.f.\ \cite[Theorem~7.1]{hadjisavvas2006handbook}), the desired Fenchel--Moreau identity here is equivalent to the statement that the $\mathscr A$-convexity coincides with the usual notion of convexity for nondecreasing functions defined on $\C$. We refer to \cite{moreau1970inf, singer1997abstract, hadjisavvas2006handbook} for more details on abstract convexity. Studies of increasing functions on cones from the perspective of abstract convexity include \cite{dutta2004monotonic, dutta2004monotonic2, dutta2008monotonic}.

\smallskip

The rest of the paper is organized as follows. We introduce definitions and state the main results in Section~\ref{section:def_main}. These results will be proved in Section~\ref{section:on_self_dual} and Section~\ref{section:on_perfect}. Lastly, examples of perfect cones in finite dimensions and infinite dimensions are given in Section~\ref{section:examples}.

\subsection*{Acknowledgements}
We thank Jean-Christophe Mourrat and Tim Hoheisel for helpful discussions. We are grateful to the anonymous referees for comments that helped to improve the paper significantly.

\section{Definitions and main results}\label{section:def_main}

Let $\H$ be a real Hilbert space equipped with inner product $\la \cdot,\cdot \ra$ and the associated norm $|\cdot|$. We refer to an element in $\H$ sometimes as a vector, though $\H$ can be possibly infinite-dimensional. We denote the interior, the closure and the boundary relative to $\H$ by $\mathsf{int}$, $\mathsf{cl}$, and $\mathsf{bd}$, respectively.

\subsection{Definitions related to cones}

Let $\C$ be a cone in $\H$. In this work, for simplicity, we require all cones to be convex and contain the origin. Hence, $\C$ is a cone if and only if it satisfies
\begin{align*}
    \alpha x +\beta y \in \C, \quad \forall x, y\in \C,\ \forall \alpha,\beta \geq 0.
\end{align*}

\smallskip

Naturally, $\C$ induces a preorder $\preceq$ on $\H$ given by
\begin{align*}
    x\preceq y\quad \text{if and only if}\quad y-x\in\C.
\end{align*}
We also write $x\succeq y$ if $y\preceq x$. When $\C$ is \textit{pointed}, namely $\C\cap(-\C) =\{0\}$, this preorder becomes a partial order. 
We denote by $\lspan$ and $\cspan$ the operations of taking the linear span and the closed linear span, respectively.
The dual of $\C$ with respect to $\cspan \C$ is given by
\begin{align}\label{eq:C_dual}
    \C^\odot = \{x\in \cspan \C:\ \la x,y\ra\geq 0,\, \forall y\in\C\}.
\end{align}
The cone $\C$ is said to be self-dual (with respect to $\cspan\C$) provided $\C =\C^\odot$. It is clear that a self-dual cone is closed and pointed. 

\smallskip

A subset $\F$ of a closed and pointed cone $\C$ is a \textit{face} of $\C$ if $\F$ is a cone and satisfies that
\begin{align}\label{eq:face}
    \text{if }\  0\preceq x \preceq y\  \text{ and }\ y\in \F, \quad\text{ then }x\in\F. 
\end{align}
Denote by $\F^{\vee}$ the dual cone of $\F$ in the space $\cspan \F$. The following definition is a generalization of \cite[Definition~4]{barker1978perfect} from Euclidean spaces to Hilbert spaces.
\begin{Def}\label{def:perfect}
A cone $\C$ is said to be \textit{perfect} if every face $\F$ of $\C$ satisfies
\begin{enumerate}
    \item $\F^{\vee} = \F$;\label{item:1}
    \item $\F$ has nonempty interior with respect to $\cspan\F$.\label{item:2}
\end{enumerate}

\end{Def}

\rmk{Since $\C$ is a face of itself, a perfect cone $\C$ is self-dual.} In finite-dimensions, a self-dual cone always has nonempty interior in its own span (c.f. \cite[Exercise 6.15]{bauschke2011convex}). Hence, if $\H$ is finite-dimensional, then \eqref{item:2} automatically follows from \eqref{item:1}. Compared with \cite[Definition 3]{barker1978perfect} where only \eqref{item:1} is imposed, condition \eqref{item:2} is added to ensure this non-degeneracy in infinite dimensions. In Section~\ref{section:examples}, we give two examples of perfect cones, a finite-dimensional one and an infinite-dimensional one.

\subsection{Definitions related to functions}

The domain of a function $f:\C\to(-\infty,\infty]  $ is defined as
\begin{align}\label{eq:domain_f}
    \dom f = \big\{x\in\C:\ f(x)<\infty\big\}.
\end{align}
A function $f:\C\to(-\infty,\infty]  $ is said to be $\C$-\textit{nondecreasing}  provided
\begin{align*}
    f(x)\geq f(y), \quad \forall x\succeq y\succeq 0.
\end{align*}
For any $f:\C\to(-\infty,\infty]  $, we define the \textit{monotone conjugate} of $f$ by \eqref{eq:def_f*} and the \textit{monotone biconjugate} of $f$ by \eqref{eq:f**}.
Lastly, $f$ is said to be \textit{proper} if $f$ is not identically equal to $\infty$. We denote by $\Gamma_\nearrow(\C)$ the collection of functions on $\C$ with values in $(-\infty,\infty]$ that are proper, convex, lower semicontinuous (l.s.c.), and $\C$-nondecreasing.

\smallskip

\rmk{Note that the ambient Hilbert space $\cH$ is not playing an important role. By restricting to a subspace, one can assume that $\cH=\cspan \C$ when needed.}

\subsection{Main results}

For any closed subspace $\H'\subset \H$, we denote by $\P_{\H'}$ the orthogonal projection onto $\H'$.
\begin{Th}\label{th:gen}
\rmk{Let $\C\subset \H$ be a closed and pointed cone. Assume that}
\begin{enumerate}[start=1,label={\rm{(H\arabic*)}}]
\item \label{H:face} every face $\F$ of $\C$ is closed and has nonempty interior with respect to $\cspan\F$;
        \item \label{H:dual_face} 
        for every face $\F$ of $\C$, the dual cone $\F^\vee$ of $\F$ in the space $\cspan\F$ is contained in $\P_{\cspan\F}\,(\C^\odot)$.
\end{enumerate}
Let $f:\C\to(-\infty, \infty]$ be proper. Then, $f=f^{**}$ if and only if $f\in \Gamma_\nearrow(\C)$. 
\end{Th}

If $f=f^{**}$, then it is easy to see $f\in \Gamma_\nearrow(\C)$ necessarily. The nontrivial part is the sufficient condition for $f=f^{**}$. As a special case, the following holds.

\begin{Cor}\label{cor:on_perfect}
Suppose that $\C$ is a perfect cone. Let $f:\C\to(-\infty, \infty]$ be proper. Then, $f=f^{**}$ if and only if $f\in \Gamma_\nearrow(\C)$. 
\end{Cor}

\rmk{
In finite dimensions, given that $\C$ is closed, every face $\F$ is automatically closed (see \cite[Corollary~18.1.1]{rockafellar1970convex}), and the second half of \ref{H:face} also holds. The proposition below shows that \ref{H:dual_face} is nearly sharp in finite dimensions.}
\begin{Prop}\label{prop:dual_face}
Assume that $\H$ is finite-dimensional. If $f=f^{**}$ for all $f\in\Gamma_\nearrow(\C)$, then every face $\F$ of $\C$ satisfies $\F^\vee\subset\cl(\P_{\lspan \F}(\C^\odot))$.
\end{Prop}

We believe that our results can be extended to more general scenarios. Here, we stick to the current setting for simplicity of presentation.

\section{Preliminaries}\label{section:on_self_dual}

In the first part of this section, we state some basic results that are needed throughout this work. In the second part, we prove Proposition~\ref{prop:dual_face}. In the last part, we prove the following result.

\begin{proposition}\label{prop:biconj_int}
Suppose that $\C$ is closed and pointed. Let $f:\C\to (-\infty,\infty] $ satisfy $\itr\dom f\neq \emptyset$. Then $f=f^{**}$ if and only if $f\in \Gamma_\nearrow(\C)$.
\end{proposition}

\subsection{Basic results of convex analysis}

For $a\in\H$ and $\nu\in\R$, we define the affine function $ L_{a,\nu}$ with slope $a$ and translation $\nu$ by
\begin{align}\label{eq:affine}
    L_{a,\nu}(x)=\la a,x\ra+\nu,\quad \forall x\in \H.
\end{align}
For a function $f:\E\to(-\infty,\infty]$ defined on a subset $\E\subset \H$, we can extend it in the standard way to $f:\H\to(-\infty,\infty]$ by setting $f(x)=\infty$ for $x\not\in\E$. For $f:\H\to(-\infty,\infty]$, we define its domain by
\begin{align*}
    \dom f = \big\{x\in\H:\ f(x)<\infty\big\}.
\end{align*}
Note that by the standard extension, the above definition is equivalent to \eqref{eq:domain_f} where only functions defined on $\C$ are considered. Henceforth, we shall not distinguish functions defined on $\C$ from their standard extensions to $\H$. Denote by $\Gamma_0(\E)$ the collection of proper, convex and l.s.c.\ functions from $\E\subset \H$ to $(-\infty,\infty]$. In particular, when $\C$ is closed, the collection $\Gamma_\nearrow(\C)\subset \Gamma_0(\C)$  can be viewed as a subcollection of $\Gamma_0(\H)$.

\smallskip

For $f:\H\to(-\infty,\infty]$ and each $x\in\H$, we define the subdifferential of $f$ at $x$ by
\begin{equation}\label{eq:def_subdiff}
    \partial f(x)=\big\{u\in\H:f(y)\geq f(x)+\la y-x,u\ra,\,\forall y\in\H\big\}.
\end{equation}
The effective domain of $\partial f$ is defined to be
\begin{align*}
    \dom \partial f = \big\{x\in\H:\ \partial f(x)\neq\emptyset\big\}.
\end{align*}

\smallskip

We now list some lemmas needed in our proofs. 
\begin{Lemma}\label{lemma:itrconvex}
For a convex set $A\subset \H$, if $y\in\mathsf{cl}\,A$ and $y'\in\mathsf{int}\,A$, then $\lambda y+(1-\lambda)y'\in\mathsf{int}\,A$ for all $\lambda \in[0,1)$.
\end{Lemma}
\begin{Lemma}\label{lemma:itrsubdif}
For $f\in\Gamma_0(\H)$, it holds that $\itr\dom f\subset \mathsf{dom}\, \partial f\subset \dom f$.
\end{Lemma}
\begin{Lemma}\label{lemma:lineconv}
Let $f\in\Gamma_0(\H)$, $x\in\H$ and $y\in\mathsf{dom}\,f$. For every $\alpha\in(0,1)$, set $x_\alpha=(1-\alpha)x+\alpha y$. Then $\lim_{\alpha\to 0} f(x_\alpha)=f(x)$.
\end{Lemma}

\begin{Lemma}\label{lemma:supsubdif}
Suppose that $\C$ is closed. Let $f\in\Gamma_0(\C)$, $x\in\C$ and $u\in\C^\odot$. If $u\in\partial f(x)$, then $f^*(u)=\la x,u\ra-f(x)$.
\end{Lemma}

\begin{Lemma}\label{lemma:hyperplane}
For $f\in\Gamma_0(\C)$ and $x\in\C$, we have
\begin{align}\label{eq:upper_envelope}
    f^{**}(x)=\sup L_{a,\nu}(x)
\end{align}
where the supremum is taken over
\begin{align}\label{eq:sup_set}
    \{(a,\nu)\in\C^\odot\times\R: L_{a,\nu}\leq f\text{ on }\C\}.
\end{align}

\end{Lemma}

Lemmas \ref{lemma:itrconvex}, \ref{lemma:itrsubdif}, and  \ref{lemma:lineconv} can be found in \cite{bauschke2011convex} as Propositions~3.35, 16.21, and 9.14, respectively. 
For completeness, let us quickly prove Lemma~\ref{lemma:supsubdif} and Lemma~\ref{lemma:hyperplane}.

\begin{proof}[Proof of Lemma~\ref{lemma:supsubdif}]
By the standard extension, we have $f\in\Gamma_0(\H)$. Invoking \cite[Theorem 16.23]{bauschke2011convex}, it is classically known that
\begin{align*}
    \sup_{z\in\H}\big\{\la z, u \ra- f(z)\big\} = \la x, u \ra -f(x).
\end{align*}
By assumption, we know $x\in\dom\partial f$. Hence, Lemma~\ref{lemma:itrsubdif} implies $x\in \dom f$ and thus the right hand side of the above equation is finite. Then, the supremum on the left must also be finite. On the other hand, by the extension, we have $f(z)=\infty$ if $z\not\in \C$, which yields
\begin{align*}
    \sup_{z\in\H}\big\{\la z, u \ra- f(z)\big\} = \sup_{z\in\C}\big\{\la z, u \ra- f(z)\big\} = f^*(u).
\end{align*}
\end{proof}

\begin{proof}[Proof of Lemma~\ref{lemma:hyperplane}]
For each $y\in\C^\odot$,
\begin{align*}
    L_{y,\, -f^*(y)}(x) = \la y,x\ra - f^*(y),\quad \forall x \in \C.
\end{align*}
is an affine function with slope $y\in\C^\odot$. By \eqref{eq:def_f*}, we can see that $L_{y,\,-f^*(y)}\leq f$ on $\C$. In view of the definition of $f^{**}$ in \eqref{eq:f**}, we have $f^{**}(x)\leq \sup L_{a,\nu}(x)$ for all $x\in\C$ where the $\sup$ is taken over the collection in \eqref{eq:sup_set}.

\smallskip

For the other direction, for $L_{a, \nu}$ satisfying \rmk{$a \in \C^\odot$} and $L_{a, \nu}\leq f$, we have
\begin{align*}
    \la a, x\ra +\nu \leq f(x),\quad\forall x\in\C.
\end{align*}
Rearranging and taking supremum in $x\in\C$, we get $f^*(a)\leq -\nu$. This yields
\begin{align*}
    L_{a, \nu}(x)\leq \la a, x\ra -f^*(a)\leq f^{**}(x),
\end{align*}
which implies $ \sup L_{a,\nu}(x)\leq f^{**}(x)$.
\end{proof}

\subsection{Proof of Proposition~\ref{prop:dual_face}}
We first prove the following lemma.
\begin{Lemma}\label{lem:<w,x>=sup<v,x>}
For $w\in \F^\vee$, the function $f:\C\to \R\cup\{\infty\}$ given by \begin{align*}
    f(x) =
    \begin{cases}
    \la w, x\ra  &\quad \text{if}\ x \in \F,
    \\
    \infty  &\quad \text{if}\ x \not\in \F.
    \end{cases}
\end{align*}
belongs to $\Gamma_\nearrow(\C)$. Moreover, if $f=f^{**}$, then
\begin{align}\label{eq:f=sup<v,x>}
    \la w,x\ra = \sup\, \la v,x\ra,\quad\forall x \in \F
\end{align}
where the supremum is taken over
\begin{align}\label{eq:set_of_v}
    \{v\in \cl(\P_{\cspan \F}(\C^\odot)): w-v\in \F^\vee\}.
\end{align}
\end{Lemma}

\begin{proof}
It is clear that $f$ is convex, proper, and l.s.c. To show $f$ is $\C$-nondecreasing, let $0\preceq x\preceq y$. Note that this implies $0\preceq y-x\preceq y$. If $y\in\F$, by \eqref{eq:face} in the definition of faces, we have $x\in\F$ and $y-x\in\F$. This along with $w\in\F^\vee$ yields $f(x)\leq f(y)$. If $y\not\in\F$, then $f(x)\leq \infty = f(y)$. This verifies $f\in \Gamma_\nearrow(\C)$.

Now, we want to show \eqref{eq:f=sup<v,x>}. By Lemma~\ref{lemma:hyperplane},
\begin{align}\label{eq:f=<a,x>+nu}
    f(x) = \sup\{ \la a, x\ra +\nu\},\quad\forall x\in \F
\end{align}
where the supremum is taken over all $(a,\nu)\in\C^\odot\times\R$ satisfying
\begin{align*}
    \la w, y \ra \geq \la a,y\ra + \nu,\quad\forall y \in \F,
\end{align*}
which is equivalent to
\begin{align*}
    \la w-a, \lambda y \ra \geq \nu,\quad\forall \lambda \geq 0 ,\ y\in \F.
\end{align*}
Setting $\lambda = 0$ and $\lambda \to \infty$ yield, respectively,
\begin{gather*}
    \nu\leq 0, \qquad\text{and}\qquad \la w-a,y\ra \geq 0,\quad\forall y\in\F.
\end{gather*}
For every such $(a,\nu)$, setting $v=\P_{\cspan \F}(a)$ (which gives $\la v,y\ra=\la a,y\ra$ for all $y\in\F$), we thus obtain
\begin{align*}
    \la w, y\ra \geq \la v,y\ra \geq \la a,y\ra+\nu,\quad\forall y\in\F.
\end{align*}
In particular, this implies that $v$ belongs to the set in \eqref{eq:set_of_v}. Hence, in view of \eqref{eq:f=<a,x>+nu}, we conclude
\begin{align*}
    f(x)\leq \sup\,\la v,x\ra ,\quad\forall x\in\F
\end{align*}
where the supremum is over the set in~\eqref{eq:set_of_v}. On the other hand, for every $v$ in the set in \eqref{eq:set_of_v}, we have $f(x)=\la w,x\ra \geq \la v,x\ra$ for all $x\in\F$. This completes the proof of~\eqref{eq:f=sup<v,x>}.
\end{proof}

Now, we are ready to prove Proposition~\ref{prop:dual_face}.
Since $\H$ is finite-dimensional, we have $\lspan \F=\cspan\F$.
We argue by contradiction and assume that there is $w\in \F^\vee \setminus \cl(\P_{\lspan \F}(\C^\odot)) $. Then, by separation theorems, 
there are $\eps>0$ and $z\in \lspan \F$ such that
\begin{align}\label{eq:wz>uz}
    \la w, z\ra \geq \la u,z\ra +\eps ,\quad\forall u \in \P_{\lspan \F}(\C^\odot).
\end{align}
By \cite[Proposition~6.4~(i)]{bauschke2011convex} and the fact that $\F$ is a cone, we have $\lspan\F=\F-\F$. Hence, there are $x,y\in\F$ such that $z=x-y$. 
By Lemma~\ref{lem:<w,x>=sup<v,x>}, we can find $v$ from the set in \eqref{eq:set_of_v} such that
\begin{align*}
    \la w,x\ra <\la v,x\ra +\eps.
\end{align*}
On the other hand, since $\la w,y\ra \geq\la v,y\ra $ due to \eqref{eq:set_of_v}, we obtain from \eqref{eq:wz>uz} that
\begin{align*}
    \la w, x\ra \geq \la v,x\ra +\eps,
\end{align*}
contradicting the previous display.

\subsection{Proof of Proposition~\ref{prop:biconj_int}}\label{section:pf_prop}
Let $f:\C\to (-\infty,\infty] $ be proper and satisfy $\itr\dom f\neq \emptyset$.
It is clear from \eqref{eq:f**} that $f^{**}$ is convex, l.s.c., and $\C$-nondecreasing. Therefore, assuming $f=f^{**}$ and that $f$ is proper, we have $f\in \Gamma_\nearrow(\C)$. 

\smallskip

From now on, we assume $f\in \Gamma_\nearrow(\C)$ and prove the converse. For convenience, we write $\Omega = \dom f$. The plan is to prove the identity $f=f^{**}$ first on $\itr\Omega$, then on $\cl\Omega$, and finally on the entire $\C$.

\subsubsection{Analysis on $\itr\Omega$}

Let $x\in \itr \Omega$. By Lemma~\ref{lemma:itrsubdif}, we know $\partial f(x)$ is not empty. For each $v\in \C$, there is $\eps>0$ small so that $x -\eps v \in \Omega$. For each $u\in \partial f(x)$, by the definition of subdifferentials and nondecreasingness, we have
\begin{equation*}
    \la v,u\ra\geq \frac{1}{\eps}\big( f(x)-f(x-\eps v)\big)\geq 0,
\end{equation*}
which implies
\begin{align}\label{eq:partial_f(x)_in_C}
    \emptyset \neq \partial f(x)\subset \C^\odot,\quad\forall x\in\itr\Omega.
\end{align}
Invoking Lemma~\ref{lemma:supsubdif}, from \eqref{eq:partial_f(x)_in_C} we can deduce
\begin{align*}
    f(x)\leq\sup_{y\in\C^\odot}\{\la y,x\ra-f^*(y)\}=f^{**}(x).
\end{align*}
On the other hand, from \eqref{eq:f**},
it is easy to see that
\begin{align}\label{eq:f>f**}
    f(x)\geq f^{**}(x),\quad\forall x\in\C.
\end{align}
Hence, we obtain
\begin{align*}
    f(x)=f^{**}(x),\quad\forall x\in\itr\Omega.
\end{align*}

\bigskip

\subsubsection{Analysis on $\mathsf{cl}\,\Omega$}
Let $x\in\cl\Omega$ and choose $y\in\itr\Omega$. Then, $x_\alpha= (1-\alpha)x+\alpha y$ belongs to $\itr\Omega$ for every $\alpha \in (0,1]$ due to Lemma~\ref{lemma:itrconvex}. By the result on $\itr\Omega$, we have
\begin{align*}
    f(x_\alpha) = f^{**}(x_\alpha).
\end{align*}
Then, $x_\alpha$ belongs to $\dom f$ and $\dom f^{**}$. Applying Lemma~\ref{lemma:lineconv} and sending $\alpha \to 0$, we conclude that \begin{align}\label{eq:f=f**_on_cl}
    f(x)=f^{**}(x),\quad\forall x\in \cl\Omega.
\end{align}

\subsubsection{Analysis on $\C$}

Due to \eqref{eq:f=f**_on_cl}, we only need to consider vectors outside $\cl\Omega$. Let $x\in \C\setminus\mathsf{cl}\,\Omega$, and we have $f(x)=\infty$. Since $f$ is proper and $\C$-nondecreasing, we must have $0\in\Omega$.
By this, $x\notin\mathsf{cl}\,\Omega$ and the convexity of $\cl\Omega$, we must have
\begin{equation}\label{eq:lambda'<1}
    \lambda'=\sup\{\lambda\in[0,\infty):\lambda x\in\mathsf{cl}\,\Omega\}<1.
\end{equation}
We set
\begin{align}\label{eq:x'}
    x'=\lambda' x.
\end{align}
Then, we have that $x'\in\bd\Omega$ and $\lambda x'\notin \mathsf{cl}\,\Omega$ for all $\lambda>1$. 

\smallskip
We need to discuss two cases: $x'\in\Omega$ or not.

\smallskip
In the second case where $x'\notin \Omega$, we have $f(x')=\infty$. Due to $x'\in\cl\Omega$ and \eqref{eq:f=f**_on_cl}, we have $f^{**}(x')=\infty$. On the other hand, by \eqref{eq:f=f**_on_cl} and the fact that $0\in\Omega$, we have $f^{**}(0)=f(0)$ and thus $0\in \dom f^{**}$. The convexity of $f^{**}$ implies that
\begin{align*}
    \infty = f^{**}(x')\leq \lambda' f^{**}(x)+ (1-\lambda')f^{**}(0).
\end{align*}
Hence, we must have $f^{**}(x)=\infty$ and thus $f(x)=f^{**}(x)$ for such $x$.

\bigskip

We now consider the case where $x'\in\Omega$. For every $y\in\H$, the outer normal cone to $\Omega$ at $y$ is defined by
\begin{equation}\label{eq:ndef}
    \mathbf n (y)=\{z\in \H:\la z, y'-y\ra\leq 0, \,\forall y'\in\Omega\}.
\end{equation}
We need the following result.

\begin{Lemma}\label{lemma:outernormal}
Assume $\itr\Omega\neq \emptyset$. For every $y\in \Omega\setminus\mathsf{int}\,\Omega$ satisfying $\lambda y\notin \mathsf{cl}\,\Omega$ for all $\lambda>1$, there is $z\in \nn(y)\cap \C^\odot$ such that $\la z, y\ra >0$.

\end{Lemma}
\begin{proof}

Fix $y$ satisfying the condition. It can happen that $y\not\in\itr\C$, and we want to approximate $y$ by a point in $\mathsf{bd}\,\Omega\cap \mathsf{int}\,\C$. The following construction is illustrated in Figure~\ref{fig}. For every open ball $B\subset \H$ centered at $y$, there is some $\lambda>1$ such that $y'=\lambda y\in \C\cap (B\setminus \mathsf{cl}\,\Omega)$. Due to $\mathsf{int}\,\Omega\neq \emptyset$ and $y\in\Omega$, by Lemma~\ref{lemma:itrconvex}, there is some $y''\in B\cap \mathsf{int}\,\Omega\subset \mathsf{int}\,\C$. For $\rho\in[0,1]$, we set
\begin{equation*}
    y_\rho= \rho y'+(1-\rho)y''\in B.
\end{equation*}
Then, we take
\begin{align*}
    \rho_0=\sup\{\rho\in[0,1]:y_\rho\in\mathsf{int}\,\Omega\}.
\end{align*}
Since $y'\notin\mathsf{cl}\,\Omega$, we must have $\rho_0<1$. It can be seen that $y_{\rho_0}\in\mathsf{cl}\,\Omega\setminus\mathsf{int}\,\Omega$ and thus $y_{\rho_0}\in B\cap\mathsf{bd}\,\Omega$. Due to $y'\in\C$, $y''\in\mathsf{int}\,\C$ and Lemma~\ref{lemma:itrconvex}, we have $y_{\rho_0}\in \mathsf{int}\,\C$. In summary, we obtain $y_{\rho_0}\in B\cap\mathsf{bd}\,\Omega\cap \mathsf{int}\,\C$.

\begin{figure}
    \centering
    
    \begin{tikzpicture}
        \node (0) at (0, 0) {};
        \node (1) at (11, 0) {};
        \node (3) at (7, 0) {};
        \node (4) at (8, 8) {};
        \node (5) at (6, 6) {};
        \node (6) at (6.955, 4) {};
        \node (7) at (5.75, 2.25) {};
        \node (8) at (9.25, 3.75) {};
        \node (9) at (7.2, 2.875) {};
        \node (10) at (7.6, 2.875) {};
        \node (11) at (3, 1) {$\Omega$};
        \tkzLabelPoint[right](6){$y$}
        \tkzLabelPoint[below](7){$y''$}
        \tkzLabelPoint[below](8){$y'$}
        \tkzLabelPoint[below](10){$y_{\rho_0}$}
        \foreach \n in {6,7,8,9}
        \node at (\n)[circle,fill,inner sep=1.5pt]{};
        \draw[black](6.955, 4) circle (3);
        \draw[thick] (0.center) to (1.center);
        \draw[thick] (0.center) to (4.center);
        \draw[thick] [bend left, looseness=0.75] (5.center) to (3.center);
        \draw[dotted,thick] (7) to (8);
\end{tikzpicture}
    \caption{Construction of $y_{\rho_0}$.}
    \label{fig}
\end{figure}
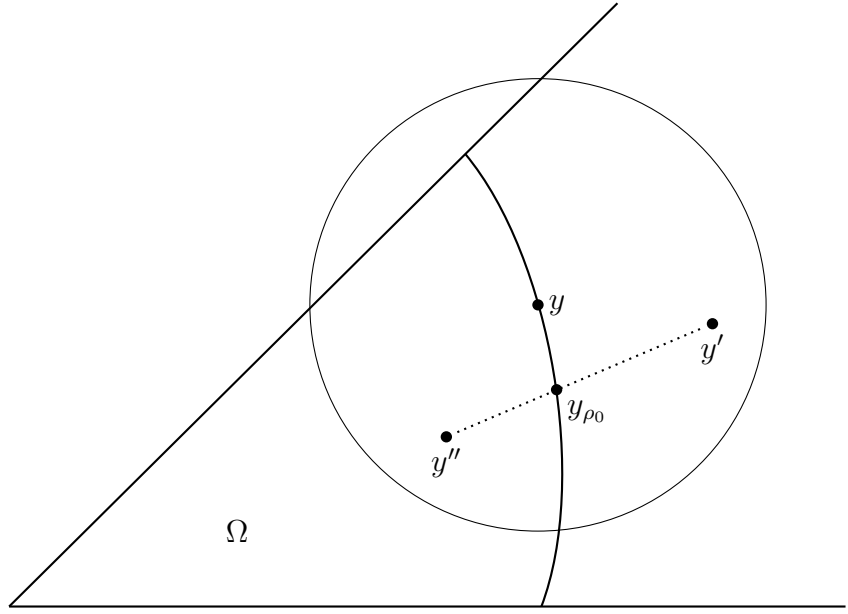{}

By this construction and varying the size of the open balls centered at $y$, we can find a sequence $\{y_n\}_{n=1}^\infty$ such that
\begin{gather}
    y_n\in \itr \C,\label{eq:y_n_int}\\
    y_n\in\mathsf{bd}\,\Omega,\label{eq:y_n_bd}\\
    \lim_{n\rightarrow \infty}y_n=y.\label{eq:y_n_lim}
\end{gather}

\smallskip

Fix any $n$. By \eqref{eq:y_n_int}, there is $\delta >0$ such that
\begin{align}\label{eq:y_n+B_in_C}
    y_n+ B(0,2\delta) \subset\C.
\end{align}
Here, for $a\in\H,r>0$, we write $B(a,r)=\{z\in\H:\ |z-a|<r\}$.
For each $\eps\in(0,\delta)$, due to \eqref{eq:y_n_bd}, we can also find $y_{n,\eps}$ such that
\begin{gather}
    y_{n,\eps}\in\Omega\label{eq:y_n,eps_itr},\\
    |y_{n,\eps}-y_n|<\eps.\label{eq:y_n,eps-y_n}
\end{gather}
This and \eqref{eq:y_n+B_in_C} imply that
\begin{align}
    y_{n,\eps}-a \in \C,\quad \forall \eps \in(0,\delta),\ a \in B(0,\delta).\label{eq:y_n,eps-a}
\end{align}

\smallskip  

By $\C$-nondecreasingness, \eqref{eq:y_n,eps_itr} and \eqref{eq:y_n,eps-a}, we can see
\begin{align*}
    y_{n,\eps}-a\in\Omega,\quad \forall \eps\in(0,\delta),\ a\in \C\cap B(0,\delta).
\end{align*}
Due to \eqref{eq:y_n_bd} and $\itr\Omega\neq \emptyset$, we have that $\mathbf n(y_n)$ contains some nonzero vector $z_n$ (see \cite[Proposition~6.45]{bauschke2011convex} together with \cite[Proposition~6.23~(iii)]{bauschke2011convex}). The definition of the outer normal cone in \eqref{eq:ndef} yields
\begin{align*}
    \la z_n,y_{n,\eps}-a-y_n\ra\leq 0,
\end{align*}
which along with \eqref{eq:y_n,eps-y_n} implies
\begin{equation*}
    \la z_n, a\ra\geq -|z_n|\eps.
\end{equation*}
Sending $\eps\rightarrow 0$ and varying $a\in\C\cap B(0,\delta)$, we conclude that
\begin{align*}
    z_n\in\mathbf n(y_n)\cap \C^\odot,\quad \forall n.
\end{align*}

\smallskip

Now for each $n$, we rescale $z_n$ to get $|z_n|=1$. Since $\C^\odot\cap \cl B(0,1)$ is convex, closed, and bounded, invoking the Banach--Alaoglu--Bourbaki theorem and the Eberlein--\v{S}mulian theorem, by passing to a subsequence, we can assume that there is $z\in\C^\odot$ such that $z_n$ converges weakly to $z$. By $z_n\in\mathbf n(y_n)$, we get
\begin{equation}\label{eq:z_n,w-y_n}
    \la z_n,w-y_n\ra\leq 0,\quad \forall w\in\Omega. 
\end{equation}
The weak convergence of $\{z_n\}_{n=1}^\infty$ along with the strong convergence in \eqref{eq:y_n_lim} implies
\begin{align*}
    \lim_{n\rightarrow \infty}\la z_n,w-y_n\ra=\la z,w-y\ra,\quad \forall w\in\Omega.
\end{align*}
The above two displays yield $z\in\mathbf n(y)\cap \C^\odot$.

\smallskip

Then, we show $\la z, y\ra >0$.
Fix some $x_0\in\itr\Omega$ and some $\eps>0$ such that $B(x_0,2\eps)\subset \Omega$. Let $y_n$ and $z_n$ be given as in the above. Due to $|z_n|=1$, we have
\begin{align*}
    x_0-\eps z_n\in\Omega\subset \C,
\end{align*}
which along with the fact that $z_n\in\C^\odot$ implies that
\begin{align*}
    \la x_0-\eps z_n,z_n\ra \geq 0
\end{align*}
and thus $\la x_0,z_n\ra\geq \eps$. Using $z_n\in \nn(y_n)$, we obtain
\begin{align*}
    \la y_n,z_n\ra \geq \la x_0,z_n\ra \geq \eps.
\end{align*}
Passing to the limit, we conclude that $\la z,y\ra>0$ completing the proof.
\end{proof}

We now go back to our main proof and apply Lemma~\ref{lemma:outernormal} to  $x'\in\Omega$. Hence, there is $z\in \C^\odot$ such that
\begin{gather}\label{eq:znx'}
\la z, w-x'\ra\leq 0,\quad\forall w\in\Omega,\\
\label{eq:zx'>0}\la z,x'\ra>0.
\end{gather}

\smallskip

By \eqref{eq:f=f**_on_cl} and Lemma~\ref{lemma:hyperplane} (or the simple fact that $f\geq f(0)$), there is an affine function $ L_{a,\nu}$ with $a\in\C^\odot$ and $\nu\in\R$ such that $f\geq  L_{a,\nu}$. For each $\rho\geq 0$, define
\begin{equation*}
     \cL_\rho= L_{a+\rho z,\ \nu-\rho\la z,x'\ra}.
\end{equation*}
Due to \eqref{eq:znx'}, we can see that
\begin{align*}
     \cL_\rho(w)&= L_{a,\nu}(w)+ \rho \la z, w-x'\ra \leq  L_{a,\nu}(w)\leq f(w),\quad\forall w\in\Omega.
\end{align*}
Since we know $f\big|_{\C\setminus\Omega}=\infty$, the inequality above gives us
\begin{align}\label{eq:cL_rho<f}
    \cL_\rho\leq f,\quad\forall \rho\geq 0.
\end{align}
Evaluating $\cL_\rho$ at $x$ and using \eqref{eq:x'},
we have
\begin{align*}
     \cL_\rho(x)&= L_{a,\nu}(x)+\rho\la z,x-x'\ra= L_{a,\nu}(x)+\rho(\lambda'^{-1}-1)\la z,x'\ra.
\end{align*}
By \eqref{eq:lambda'<1} and \eqref{eq:zx'>0}, we obtain
\begin{align*}
    \lim_{\rho\rightarrow\infty} \cL_\rho(x)=\infty.
\end{align*}
This along with \eqref{eq:cL_rho<f}, Lemma~\ref{lemma:hyperplane} and \eqref{eq:f>f**} implies
\begin{align*}
    f(x)=f^{**}(x)\quad\forall x \in \C\setminus\cl\Omega.
\end{align*}
In view of this and \eqref{eq:f=f**_on_cl}, we have completed the proof of Proposition~\ref{prop:biconj_int}.

\section{Proof of Theorem~\ref{th:gen}}\label{section:on_perfect}

We devote this section to the proof of Theorem~\ref{th:gen}. As commented in the beginning of the proof of Proposition~\ref{prop:biconj_int} in Section~\ref{section:pf_prop}, assuming $f=f^{**}$, we have $f\in\Gamma_\nearrow(\C)$. 

\smallskip

Now, assuming \ref{H:face}, \ref{H:dual_face} and $f\in\Gamma_\nearrow(\C)$, we want to prove $f=f^{**}$. \rmk{By restricting to $\cspan\C$, we can assume that $\cH = \cspan\C$.} Again, we write $\Omega = \dom f$ which is a nonempty subset of $\C$.
Let us introduce
\begin{align*}
    \F_\Omega = \big\{\lambda y:\  \lambda \geq 0,\, y\in \Omega\big\}.
\end{align*}
We will first show that $f=f^{**}$ holds on $\F_\Omega$ and then on $\C$.

\subsection{Identity on \texorpdfstring{$\F_\Omega$}{F\_Omega}}
We prove $f=f^{**}$ on $\F_\Omega$. The idea is to show $\Omega$ has nonempty interior relative to $\F_\Omega$ and apply Proposition~\ref{prop:biconj_int} to $f$ restricted to $\F_\Omega$. Some properties of $\F_\Omega$ are needed and they are stated and proved in the two lemmas below.

\begin{Lemma}\label{lemma:F_Omega_face}
The set $\F_\Omega$ is a face of $\C$.
\end{Lemma}

\begin{proof}
Recall the definition of a face above Definition~\ref{def:perfect}. Since in this work, we require cones to be convex, to show $\F_\Omega$ is a face, we start by checking it is convex. Note that for any $x_1, x_2 \in \F_\Omega$, there are $\lambda_1, \lambda_2\geq 0$ and $y_1,y_2\in\Omega$ such that
$x_i = \lambda_iy_i$ for $i=1,2$. We can choose $\mu>0$ large enough so that $\frac{\lambda_i}{\mu}y_i\preceq y_i$ for both $i$. Hence, by the $\C$-nondecreasingness of $f$, we have $\frac{\lambda_i}{\mu}y_i\in\Omega$ for both $i$. Then, for each $\alpha\in [0,1]$, it holds that
\begin{align*}
    \alpha x_1 + (1-\alpha)x_2 = \mu\bigg(\alpha \frac{\lambda_1}{\mu}y_1 + (1-\alpha) \frac{\lambda_2}{\mu}y_2\bigg).
\end{align*}
By the convexity of $\Omega$, we have $\alpha \frac{\lambda_1}{\mu}y_1 + (1-\alpha) \frac{\lambda_2}{\mu}y_2\in\Omega$. Hence, we  conclude that $\alpha x_1+ (1-\alpha)x_2\in\F_\Omega$, which implies that $\F_\Omega$ is convex. Then, it is easy to see $\F_\Omega$ is a cone.

\smallskip

Now let $0\preceq x \preceq y$ and $y\in \F_\Omega $. By definition, there is $\mu> 0$ such that $\mu y \in \Omega$. We can deduce that $0\preceq \mu x \preceq \mu y$. Again, the $\C$-nondecreasingness implies $\mu x \in \Omega$ and thus $x\in \F_\Omega$.
\end{proof}

\begin{Lemma}\label{lemma:F_Omega_itr}
Assume \ref{H:face}. The subset $\Omega$ has nonempty interior with respect to the space $\cspan\F_\Omega$.
\end{Lemma}
\begin{proof}
For positive integers $m,n\in\N_+$, we set $E_{m,n}=\{mx\in\H:f(x)\leq n\}$ which is the level set $\{f\leq n\}$ scaled by $m$. 
We want to show
\begin{align}\label{eq:union_identity}
    \F_\Omega=\bigcup_{m,n\in\N_+}E_{m,n}.
\end{align}
For each $x\in \F_\Omega$, there is $\mu>0$ such that $y=\mu x \in \Omega$. Then, there is $n\in\N_+$ such that $ y \in \{f\leq n\}$. Choose $m\in\N$ to satisfy $m\mu\geq 1$. Since $f$ is $\C$-nondecreasing and $0\preceq \frac{1}{m\mu }y \preceq y$, it yields that $\frac{1}{m\mu }y \in \{f\leq n\}$, which implies that $x\in E_{m,n}$. The other direction is easy by the definition of $\F_\Omega$. Therefore, we have verified \eqref{eq:union_identity}.

\smallskip

Since $f$ is l.s.c., we know that every $E_{m,n}$ is closed. As a closed subspace of $\H$, the space $\cspan\F_\Omega$ is complete. On the other hand, by \ref{H:face}, the face $\F_\Omega$ also has nonempty interior in $\cspan\F_\Omega$. Hence, invoking the Baire category theorem (see \cite[Section~10.2]{royden2010real}) and taking \eqref{eq:union_identity} into account, we can deduce that there is a pair $m,n$ such that $E_{m,n}$ has nonempty interior in $\cspan\F_\Omega$. This implies that the  interior of $\{f\leq n\}\subset \Omega$ relative to $\cspan\F_\Omega$ is nonempty. Hence, we conclude that $\Omega$ has nonempty interior.
\end{proof}

\bigskip

Let us set $\C' = \F_\Omega$, $\H'=\cspan\F_\Omega$ and $f'$ be the restriction of $f$ to $\C'$. Since $\Omega\subset \F_\Omega$, it is immediate that $\dom f' = \Omega \subset \C'$. Also, due to $f\in\Gamma_\nearrow(\C)$, we have $f'\in\Gamma_\nearrow(\C')$. 
By \ref{H:face} and Lemma~\ref{lemma:F_Omega_face}, $\C'$ is closed and pointed in $\H'$.
Lemma~\ref{lemma:F_Omega_itr} guarantees that $\dom f'$ has nonempty interior in $\H'$. Therefore, invoking Proposition~\ref{prop:biconj_int}, we obtain
\begin{align}\label{eq:f'_biconj}
    {f'}(x) = {f'}^{*'*'}(x),\quad \forall x\in \C'.
\end{align}
Here, 
\begin{gather*}
    {f'}^{*'}(y)=\sup_{z\in\C'}\{\la z, y\ra-f'(z)\},\quad \forall y\in {\C'}^\vee,
    \\
    {f'}^{*'*'}(x)=\sup_{y\in{\C'}^\vee}\{\la y, x\ra-{f'}^{*'}(y)\},\quad \forall x\in \C',
\end{gather*}
where ${\C'}^\vee$ is the dual cone of $\C'$ in $\H'$.
Due to $\Omega\subset \C'$,
\begin{align}\label{eq:f=infty_out_C'}
    f(z)=\infty,\quad\forall z\not\in \C'.
\end{align}
By this and \eqref{eq:def_f*}, we have
\begin{align*}
    f^*(y) ={f'}^{*'}(y),\quad \forall y \in {\C'}^\vee.
\end{align*}
Using \ref{H:dual_face}, the definition of $f'$ and \eqref{eq:f=infty_out_C'}, we can see that
\begin{align*}
    \{ L_{a,\nu}:a\in{\C'}^\vee,\,\nu\in\R \text{ such that } L_{a,\nu}\leq f'\text{ on }\C'\}
    \\
    \subset \{ L_{a,\nu}:a\in\C^\odot,\,\nu\in\R \text{ such that } L_{a,\nu}\leq f\text{ on }\C\},
\end{align*}
which together with Lemma~\ref{lemma:hyperplane} implies that
\begin{align*}
    f^{**}(x)\geq {f'}^{*'*'}(x),\quad\forall x\in \C'.
\end{align*}
This along with \eqref{eq:f'_biconj} and $f=f'$ on $\C'$ yields $f^{**}\geq f$ on $\C'$. Lastly, from \eqref{eq:f>f**}, we conclude that
\begin{align}\label{eq:f**=f_F_Omega}
    f(x)=f^{**}(x),\quad\forall x\in \F_\Omega.
\end{align}

\subsection{Identity on \texorpdfstring{$\C$}{C}}

Due to \eqref{eq:f**=f_F_Omega}, we only need to show $f(x)=f^{**}(x)$ for $x\in\C\setminus\F_\Omega$. To start, we record useful properties of faces in the ensuing two lemmas.

\begin{Lemma}\label{lemma:face_on_bdy}
Let $\F$ be a face of a cone $\C\subset\H$. If $\F\neq \C$, then $\F\cap\itr\C=\emptyset$ and thus $\F\subset \bd \C$.

\end{Lemma}

\begin{proof}
Let us argue by contradiction and suppose that there is $x\in\F\cap \itr\C $. Then for every $y\in\C$, we can find $\eps>0$ small so that $x-\eps y \in \C$ and thus $0\preceq \eps y \preceq x$.
Then, the definition of faces implies that $\eps y\in\F$. Since $\F$ is a cone and $\eps>0$, we obtain $y\in\F$ which implies $\C\subset\F$ and thus $\C=\F$, contradicting the assumption that $\F\neq \C$. Therefore, the desired result holds.
\end{proof}

\begin{Lemma}\label{lemma:v}

Assume \ref{H:dual_face}. Let $\F$ be a face of $\C$. For every $x\in \C\setminus\F$, there is $v\in\C^\odot$ such that $\la v,x\ra >0$ and
\begin{align*}
    \la v,y\ra =0,\quad \forall y\in\F.
\end{align*}
\end{Lemma}

\begin{proof}

We take $\F'$ to be the intersection of all faces of $\C$ containing both $\F$ and $x$. It can be checked that $\F'$ is again a face of $\C$. Hence, $\F'$ is the minimal face containing both $\F$ and $x$.  Let us write
\begin{align}\label{eq:H'=spanF'}
    \H' =\cspan \F'
\end{align}
and denote by $\mathring\F'$ the interior of $\F'$ with respect to $\H'$. By \ref{H:face}, we have $\mathring\F'\neq \emptyset$ and that $\F'$ is closed. Since $\F$ is clearly a face of $\F'$,  Lemma~\ref{lemma:face_on_bdy} applied to $\F\subset\F'$ yields $\F\cap  \mathring \F' =\emptyset$.

\smallskip

By the Hahn--Banach separation theorem (c.f. \cite[Theorem 1.6]{brezis2010functional}), there are $\alpha\in\R$ and a nonzero vector $w\in\H'$ such that 
\begin{align}
    \la w, y\ra \leq \alpha ,\quad\forall y\in \F,\label{eq:leq_alpha}\\
    \la w, z\ra \geq  \alpha,\quad\forall z\in \mathring\F'.\label{eq:geq_alpha}
\end{align}
Since $\F'$ is closed and convex, and $\mathring\F'\neq \emptyset$, by \cite[Proposition~3.36 (iii)]{bauschke2011convex}, we have that the closure of $\mathring\F'$ is $\F'$. Hence, \eqref{eq:geq_alpha} becomes
\begin{align}
    \la w, z\ra \geq  \alpha,\quad\forall z\in \F'.\label{eq:geq_alpha_1}
\end{align}
By \eqref{eq:face}, we have $0\in\F$. Due to this and $\F\subset\F'$, using \eqref{eq:leq_alpha} and \eqref{eq:geq_alpha_1}, we must have $\alpha=0$ and
\begin{align}
    \la w, y\ra =0 ,\quad\forall y\in \F.\label{eq:=_0_F}
\end{align}
Then, \eqref{eq:geq_alpha_1} is turned into $\la w, z\ra \geq 0$ for all $z\in \F'$
which implies that
\begin{align}\label{eq:v_in_C}
    w\in {\F'}^\vee \end{align}
where ${\F'}^\vee$ is the dual cone of $\F'$ in $\H'$. Due to \ref{H:dual_face}, there is $v\in \C^\odot$ such that
\begin{align}\label{eq:vz=wz}
    \la v,z\ra =\la w, z\ra ,\quad\forall z \in \H'.   
\end{align}

\smallskip

Now, we consider the null space of the linear map $y\mapsto \la v,y\ra$ given by
\begin{align}\label{eq:def_E}
    \E = \{y\in\H:\la v, y\ra =0\}.
\end{align}
We want to show $\E\cap\F'$ is a face of $\C$. It is clear that $\E\cap\F'$ is a cone. For $y\in\E\cap\F'$ and $z\in\C$ satisfying $0\preceq z \preceq y$, by $v\in\C^\odot$, we obtain
\begin{gather*}
    \la v, y-z\ra \geq 0,\\
    \la v, z \ra \geq 0.
\end{gather*}
Due to $y\in\E$, the above two displays yield $\la v, z\ra =0$ which implies that $z\in\E$. Since $\F'$ is a face, by $0\preceq z \preceq y$ and $y\in\F'$, we also have $z\in\F'$. Hence, we have $z\in \E\cap \F'$ and thus verified that $\E\cap\F'$ is a face of $\C$.

We claim that
\begin{align}\label{eq:E_F'}
    \E\cap\F'\neq \F'.
\end{align}
Otherwise, we have $\F'\subset \E$, which due to \eqref{eq:H'=spanF'} implies that $\H'\subset \E$. However, this along with \eqref{eq:vz=wz} means that $\la w,w\ra=0$ contradicting the fact that $w\neq 0$. Hence, \eqref{eq:E_F'} is valid.

\smallskip

To conclude, we argue that
\begin{align}\label{eq:x_not_in_E}
    x\not\in \E.
\end{align}
Otherwise, since $\F'$ contains $x$ by the definition of $\F'$, we have $x \in \E\cap \F'$. From \eqref{eq:=_0_F}, \eqref{eq:vz=wz} and \eqref{eq:def_E}, we can deduce that $\F\subset\E$ and thus $\F\subset \E\cap\F'$. Therefore, $\E\cap\F'$ is a face containing both $x$ and $\F$. However, this together with \eqref{eq:E_F'} contradicts the fact that $\F'$ is chosen to be the minimal face containing $x$ and $\F$.

\smallskip

Therefore, by contradiction, we conclude that \eqref{eq:x_not_in_E} must hold. Then, by $x\in\F'$ and \eqref{eq:v_in_C}, we must have $\la v,x\ra>0$. In view of this, \eqref{eq:=_0_F} and \eqref{eq:vz=wz}, the vector $v$ satisfies all the desired properties.
\end{proof}

With these results, we resume the proof of $f=f^{**}$ on $\C\setminus\F_\Omega$. Fix any $x\in \C\setminus \F_\Omega$. For each $\rho>0$, we set
\begin{align*}
    \cL_\rho = L_{\rho v,\ f(0)},
\end{align*}
with $v\in\C^\odot$ given in Lemma~\ref{lemma:v} corresponding to this $x$ and $\F=\F_\Omega$. This lemma implies that $v$ is perpendicular to $\F_\Omega$ and thus
\begin{align*}
    \cL_\rho(y) = \rho\la v, y\ra +f(0)=f(0),\quad\forall y\in\F_\Omega.
\end{align*}
Then, the $\C$-nondecreasingness of $f$ implies that
\begin{align*}
    f(y)\geq   \cL_\rho(y),\quad\forall y\in\F_\Omega.
\end{align*}
Since we know $f=\infty$ on $\C\setminus\F_\Omega$, we obtain
\begin{align*}
    f\geq \cL_\rho,\quad\forall \rho>0.
\end{align*}
On the other hand, due to $\la v, x\ra>0$ in Lemma~\ref{lemma:v}, we have 
\begin{align*}
    \lim_{\rho\to\infty}\cL_\rho(x)=\infty=f(x).
\end{align*}
Hence, by the above two displays, \eqref{eq:f>f**} and Lemma~\ref{lemma:hyperplane}, we conclude that $f(x)=f^{**}(x)$ for  $x\in \C\setminus\F_\Omega$. This together with \eqref{eq:f**=f_F_Omega} completes the proof of Theorem~\ref{th:gen}.

\subsection{Proof of Corollary~\ref{cor:on_perfect}}
Recall the notion of perfect cones in Definition~\ref{def:perfect}. We verify that any perfect cone $\C$ satisfies \ref{H:face} and \ref{H:dual_face}. For any face $\F$ of $\C$, by Definition~\ref{def:perfect}~\eqref{item:1}, $\F$ is self-dual in $\cspan\F$ and thus closed. Hence, \ref{H:face} follows from this and Definition~\ref{def:perfect}~\eqref{item:2}. Lastly, due to $\F\subset\C$, $\F^\vee=\F$ and $\C^\vee=\C$, it is immediate that $\F^\vee\subset \P_{\cspan F}(\C^\vee)$ and thus \ref{H:dual_face} holds. Therefore, Theorem~\ref{th:gen} yields Corollary~\ref{cor:on_perfect}.

\section{Examples of perfect cones}\label{section:examples}

We show that the set of positive semidefinite matrices is a perfect cone, and that an infinite-dimensional circular cone is perfect.

\subsection{Positive semidefinite matrices}

Let $n\in\N\setminus\{0\}$ and denote by $\S^n$ the set of all $n\times n$ symmetric matrices, by $\S^n_+$ the set of all $n\times n$ positive semidefinite matrices, and by $\S^n_{++}$ the set of all $n\times n$ positive definite matrices. On $\S^n$, we define the inner product by
\begin{align*}
    \la x , y\ra = \tr(x y),\quad \forall x,y\in\S^n,
\end{align*}
where $\tr$ is the trace of a matrix. Hence, $\S^n$ is a Hilbert space with dimension $n(n+1)/2$. The goal is the following.

\begin{proposition}\label{lemma:psd_perfect}
For each positive integer $n$, the set $\S^n_+$ is a perfect cone in $\S^n$.
\end{proposition}

To start, it is well-know that $\S^n_+$ is self-dual, which is attributed often to Fej\'er (see, e.g.\ \cite[Theorem 7.5.4]{horn2012matrix}). For completeness of presentation, we prove it below.

\begin{Lemma}\label{lemma:psd}
Let $x\in\S^n$. Then, $x\in\S^n_+$ if and only if $\la x, y\ra \geq0 $ for every $y\in\S^n_+$.
\end{Lemma}

\begin{proof}
If $x \in \S^n_+$, then for any $y\in\S^n_+$ we have $\la x, y\ra =\tr(\sqrt{x}\sqrt{y}\sqrt{y}\sqrt{x})\geq 0$. For the other direction, by choosing an orthonormal basis, we may assume that $x$ is diagonal. Testing by $y\in\S^n_+$, we can show that all diagonal entries in $x$ are nonnegative and thus $x\in\S^n_+$.
\end{proof}

\begin{proof}[Proof of \rmk{Proposition}~\ref{lemma:psd_perfect}]

Given the above lemma, we only need to verify the conditions on the faces of $\S^n_+$ stated in Definition~\ref{def:perfect}. Let $\F$ be a face of $\S^n_+$.

\smallskip

The cases $\F=\{0\}$ and $\F=\S^n_+$ are trivial, so we assume $\{0\}\subsetneq\F\subsetneq \S^n_+$. Lemma~\ref{lemma:face_on_bdy} implies $\F\subset \bd\S^n_+=\S^n_+\setminus\S^n_{++}$. Set
\begin{equation}\label{max:m}
    m=\max\big\{\mathsf{rank}(z):\ z\in\F\big\},
\end{equation}
where $\mathsf{rank}(z)$ is the rank of the matrix $z$. By our assumption on $\F$, we must have $1\leq m< n$. For each $k\in\N\setminus\{0\}$, we denote by $\boldsymbol{0}_k$ the $k\times k$ zero matrix. Due to \eqref{max:m}, there is $x\in\F$ with $\mathsf{rank}(x)=m$. By fixing a suitable orthonormal basis, we may assume 
\begin{align}\label{eq:x=diag}
    x=\diag (\lambda_1,\lambda_2,\cdots,\lambda_m,\boldsymbol{0}_{n-m}),
\end{align}
where $\lambda_j>0$ for all $1\leq j\leq m$.

\smallskip

Let us consider the following set
\begin{align}\label{eq:E_matrix}
    \E = \Big\{\diag (y^\circ,\boldsymbol{0}_{n-m}):\quad y^\circ\in\S^m_+\Big\}\subset \S^n_+.
\end{align}
We now show $\E = \F$. First, we want to prove $\F\subset\E$. In other words, we claim that for every $y\in\F$, there is $y^\circ\in\S^m_+$ such that
\begin{equation}\label{eq:ydiagform}
    y=\diag (y^\circ,\boldsymbol{0}_{n-m}).
\end{equation}
Let us argue by contradiction. Suppose that \eqref{eq:ydiagform} does not hold for all $y\in\F$, then we can find $y\in\F$ with $y_{jk}\neq 0$ for some $j>m$ or $k>m$. Assuming the former without loss of generality, we compute $v^\intercal y v$ for $v=te^j+e^k$ and vary $t\in \R$ where $e^j$ and $e^k$ belong to the standard basis for $\R^n$. Then, due to $y\in \S^n_+$, we must have $y_{jj}>0$. By reordering the basis, we may assume $j=m+1$, and thus
\begin{align}\label{eq:y_m+1>0}
    y_{m+1,\, m+1}>0.
\end{align}
Let $\hat y=(y_{ij})_{1\leq i,\,j\leq m+1}\in\S^{m+1}_+$, and we define $\hat x$ similarly. Then, we want to show $\rank(\hat x +\hat y)=m+1$.
Let $v\in\R^{m+1}\setminus\{0\}$. If $v_j\neq 0$ for some $1\leq j\leq m$, then we have
\begin{align*}
    v^\intercal (\hat x+\hat y)v \geq v^\intercal \hat x v>0.
\end{align*}
The last inequality follows from \eqref{eq:x=diag}. If $v_j=0$ for all $1\leq j\leq m$, then due to $v\neq 0$, we must have $v_{m+1}\neq 0$, and by \eqref{eq:y_m+1>0}, we get
\begin{align*}
    v^\intercal (\hat x+\hat y)v \geq v^\intercal \hat y v = y_{m+1,m+1}v^2_{m+1}>0.
\end{align*}
In conclusion, we obtain $v^\intercal (\hat x+\hat y)v>0$, which implies that $\hat x+\hat y\in\S^{m+1}_{++}$ and thus $\mathsf{rank}(x+y)\geq \mathsf{rank}(\hat x+\hat y)=m+1$. Since $\F$ is a cone, we have $x+y\in\F$. But this contradicts the maximality of $m$ as in \eqref{max:m}. Hence, every $y\in\F$ satisfies \eqref{eq:ydiagform}, and thus we verified $\F\subset \E$. 

\smallskip

Now, we turn to the proof of $\E\subset \F$. For every $y$ of the form \eqref{eq:ydiagform}, due to \eqref{eq:x=diag}, there exists a small $\eps>0$ such that $x\succeq \eps y\succeq 0 $ where the partial order $\succeq$ is induced by the cone $\S^n_+$. Indeed, such $\eps$ exists because, viewing $x,y$ as matrices in $\S^m_+$, we can choose $\eps$ sufficiently small so that the absolute values of eigenvalues of $y$ is bounded by $\min_{1\leq j\leq m}{\lambda_j}$. Recall the definition of faces above Definition~\ref{def:perfect}. Since $\F$ is a face, we must have $\eps y\in\F$ and thus $y\in\F$. Hence, we conclude $\E\subset \F$.

\smallskip

Now, we have $\F=\E$. In view of \eqref{eq:E_matrix}, we can identify $\F$ with $\S^m_+$ and $\cspan \F$ with $\S^m$.
We know that $\S^m_+$ is self-dual in $\S^m$ by Lemma~\ref{lemma:psd}, whose interior is given by $\S^m_{++}$ and thus not empty.
Therefore, all conditions on $\F$ in Definition~\ref{def:perfect} are verified.
\end{proof}

\subsection{An infinite-dimensional circular cone}

We consider a generalization of the finite dimensional circular cone $\{x\in\R^{d+1}:(x_1^2+\cdots +x_d^2)^\frac{1}{2}\leq x_0 \}$.
Let $\H= l^2(\N)$ where the elements in $l^2(\N)$ are precisely $x=(x_0,x_1,x_2,\dots)$ with $\sum_{i=0}^\infty x^2_i <\infty$. The inner product on $\H$ is given by 
\begin{align*}
    \la x,y\ra = \sum_{i=0}^\infty x_iy_i,\quad\forall x,y\in\H.
\end{align*}
We denote by $|\cdot|$ the associated norm. For each $x\in\H$, we write $x_{\geq 1}=(0,x_1,x_2,\dots)\in\H$.
We consider the following cone
\begin{align}\label{eq:C_hilbert}
    \C =\{x\in\H:\  |x_{\geq 1}|\leq x_0\}.
\end{align}
The desired result is stated below.

\begin{proposition}\label{lemma:circ_perfect}
The cone $\C$ defined in \eqref{eq:C_hilbert} is perfect in $\H$.
\end{proposition}

To prove this \rmk{proposition}, we start with the following result.

\begin{Lemma}\label{lemma:itr_C}
The interior of $\C$ is nonempty, and given by
\begin{align}\label{eq:itr_C}
    \itr\C = \{x\in \H:\ x_0 >0,\, |x_{\geq 1}|<x_0\}.
\end{align}
\end{Lemma}

\begin{proof}
Let $y$ belong to the set on the right hand side of \eqref{eq:itr_C}. Choose $\eps>0$ such that $y_0-|y_{\geq 1}|>2\eps$. Then, we want to show that, for all $x\in\H$ satisfying $|x-y|<\eps$, we have $x\in\C$. We can see that
\begin{align*}
    (x_0-y_0)^2+|x_{\geq 1}-y_{\geq 1}|^2 = |x-y|^2 <\eps^2.
\end{align*}
This yields $|x_0-y_0|<\eps$ and $|x_{\geq 1}-y_{\geq 1}|<\eps$. Now, using these, the property of $\eps$ and the triangle inequality, we get
\begin{align*}
    |x_{\geq 1}|\leq|y_{\geq 1}|+\eps<(y_0-2\eps)+\eps = y_0-\eps \leq x_0.
\end{align*}
Hence, we have $x\in \C$ and can deduce that the right side of \eqref{eq:itr_C} is contained in $\itr\C$. For the other direction, let $y\in\C$ with $|y_{\geq 1}|=y_0$. It is easy to see that every neighborhood of $y$ contains a point not in $\C$. Therefore, we conclude that \eqref{eq:itr_C} holds.
\end{proof}

\medskip

In order to prove the perfectness of $\C$, we need information about its faces. The next lemma classifies all faces of $\C$. The definition of faces are given above Definition~\ref{def:perfect}.

\begin{Lemma}\label{lemma:classify}
Under the above setting, if $\F$ is a face of $\C$, then either $\F=\C$ or there is $x\in\bd\C$ such that $\F=\{\lambda x:\ \lambda \geq 0\}$.
\end{Lemma}
\begin{proof}
It is clear that $\C$ is a face of itself.
Now we consider the case $\F\neq \C$.
If $\F=\{0\}$, then there is nothing to prove. Hence, let us further assume that there is a nonzero $x\in\F\subset\C$. In particular, due to \eqref{eq:C_hilbert}, we have $x_0>0$. Lemma~\ref{lemma:face_on_bdy} implies  $\F \subset \bd\C$. By Lemma~\ref{lemma:itr_C} and the definition of $\C$, the vector $x$ satisfies
\begin{align}\label{eq:x_>1=x_0}
    |x_{\geq 1 }|=x_0>0.
\end{align}
By definition of faces, $\F$ is a cone. Due to this and $x\in\F$, we have
\begin{align*}\F\supset\{\lambda x:\ \lambda \geq 0\}.
\end{align*}
Now, we show that the above is in fact an equality. Let $y\in \F\setminus\{0\}$. By similar arguments as above, we have $y_0>0$.
Rescaling if needed, we may assume $y_0 =x_0$. Recall that in this work, convexity is built into the definition of cones. Set $z=\frac{1}{2}(x+y)$. Using Jensen's inequality, by \eqref{eq:x_>1=x_0} and an analogous one for $y$, we obtain
\begin{align*}
    x_0^2=|z_{\geq 1}|^2=\sum_{i=1}^\infty \Big(\frac{x_i+y_i}{2}\Big)^2 \leq \sum_{i=1}^\infty \frac{x^2_i+y^2_i}{2}=x^2_0.
\end{align*}
The equality holds only if $x_i=y_i$ for all $i$, so $y=x$ and the proof is complete.
\end{proof}

\begin{proof}[Proof of \rmk{Proposition}~\ref{lemma:circ_perfect}]

We first show that $\C$ is self-dual. Recall that the dual cone is defined in \eqref{eq:C_dual} and denoted by $\C^\odot$. Let $y\in\C^\odot$ and we have
\begin{align}\label{eq:xy_geq_0}
    \la x ,y \ra \geq 0,\quad \forall x\in\C.
\end{align}
Since $(1,0,0,\dots)\in\C$, we get $y_0\geq 0$. We consider two cases depending on whether $y_0=0$ or not. Suppose $y_0=0$, for any fixed $i\geq1$, we construct $x'$ in the following way. Set $x'_0 = 1$, set $x'_i = -1$ if $y_i\geq 0$ and $x'_i=1$ if $y_i<0$, and lastly set all other entries of $x'$ to be zero. Inserting this $x'$ into \eqref{eq:xy_geq_0} and varying $i$, we can see $y=0$ and thus $y\in\C$. Now we consider the case where $y_0> 0$.
If $|y_{\geq 1}|=0$, then this immediately implies $y\in\C$. If $|y_{\geq 1}|\neq 0$, then we set $\gamma = |y_{\geq 1}|^{-1}>0$ and consider $x'$ given by
\begin{align*}
    x'_0 = y_0;\qquad x'_i = -\gamma y_iy_0.
\end{align*}
Plugging $x'$ into \eqref{eq:xy_geq_0} and using $y_0>0$, we obtain $y_0\geq |y_{\geq 1}|$ and thus $y\in\C$, which implies $\C^\vee\subset \C$. Since it is clear that $\C\subset\C^\vee$, we conclude that $\C$ is self-dual.

\medskip

To show $\C$ is perfect, it remains to check the conditions on the faces of $\C$ stated in Definition~\ref{def:perfect}.
Recall that Lemma~\ref{lemma:itr_C} ensures $\itr\C\neq \emptyset$. Hence, if $\F=\C$, then $\F$ is self-dual and has nonempty interior with respect to $\cspan \F$. Now if $\F\neq \C$, then Lemma~\ref{lemma:classify} implies $\F=\{\lambda x: \lambda\geq 0\}$, which is one-dimensional. We can identify $\cspan\F$ with $\R$ and $\F$ with $[0,\infty)$ in an isometric way. Now, it is easy to see that $\F$ is self-dual and has nonempty interior with respect to $\cspan \F$. By Definition~\ref{def:perfect}, we conclude that $\C$ given in \eqref{eq:C_hilbert} is perfect.
\end{proof}

\bibliographystyle{abbrv}

\end{document}